\RequirePackage{fix-cm}
\documentclass[smallextended]{svjour3}       
\smartqed  
\usepackage{amsmath}
\usepackage{amsfonts}
\usepackage{latexsym,amssymb}
\usepackage{graphics,hyperref}
\usepackage{graphicx}
\usepackage{multirow}
\usepackage{color}

\newcommand{\Frac}{\displaystyle \frac}

\newcommand{\Rn}{\mathbb{R}^N}
\newcommand{\R}{\mathbb{R}}

\newcommand{\Intrn}{\displaystyle \int_{\mathbb{R}^N}}

\begin{document}

\title{Mini-Max Algorithm via Pohozaev Manifold
}


\author{L.A. Maia \and D. Raom \and R. Ruviaro \and \\Y. D. Sobral 
}


\institute{L.A. Maia, R. Ruviaro, Y.D. Sobral \at
              Departamento de Matem\'atica, Universidade de Bras\'ilia, Campus Universit\'ario Darcy Ribeiro, 70910-900 Bras\'ilia-DF, Brazil \\
              Tel.: +55-61-31076478\\
              Fax: +55-61-31076481 \\
              \email{lilimaia@unb.br, ruviaro@unb.br, ydsobral@unb.br}           
           \and
            D. Raom \at
            Departamento de Engenharia Mec\^anica, Universidade de Bras\'ilia, Campus Universit\'ario Darcy Ribeiro, 70910-900 Bras\'ilia-DF, Brazil \\
              Tel.: +55-61-31075503\\
              Fax: +55-61-31075503\\
              \email{danielraoms@hotmail.com}
}
\date{}

\maketitle

\begin{abstract}
A new algorithm for solving non-homogeneous asymptotically linear and superlinear  problems is proposed. The ground state solution of the problem, which in general is obtained as a mini-max of the associated functional, is obtained as the minimum of the functional constrained to the Pohozaev manifold instead. Examples are given of the use of this method for finding numerical 
radially symmetric positive
 solutions depending on various parameters.
\vspace{0.3cm}
\\%
\end{abstract}

\section{Introduction}

The celebrated Mountain Pass Theorem of Ambrosetti and Rabinowitz \cite{AR} has been widely used in the past forty-five years for finding weak solutions of semilinear elliptic problems as critical points of an associated functional. Solutions are found on the mini-max levels of the functional. 

A numerical approach of this theorem was first introduced by Choi and McKenna in \cite{YP}. Their work showed that, when carefully implemented, the algorithm is globally convergent and leads to a solution with the required mountain pass property. 

Later, Chen, Ni and Zhou in \cite{GJW} observed that this algorithm may converge to a solution with Morse index greater or equal to two, and not to the ground state mountain pass level. In order to circumvent this limitation, they created a new algorithm based on the fact that the minimum of the associated functional constrained to the Nehari manifold is equal to the mini-max level obtained by the Mountain Pass Theorem. This equivalence follows when the nonlinear terms in the equation are superquadratic \cite{Ding,RAB,Wi}.
For the asymptotically linear problem this is not true in general. However, more recently, the ground state level was shown to be equal to the minimum of the functional restricted to the Pohozaev manifold (see Jeanjean and Tanaka \cite{JT}).

Our new algorithm is based on this analytical result. To the best of our knowledge, this is the first time in the literature that an algorithm based on this idea is constructed. 
Summarizing, the idea is to
	replace the minimization on the
	Nehari manifold (studied in the literature both analytically and numerically) by the minimization on the
	Pohozaev manifold (not studied numerically in the literature yet). The similarity of
	both concepts is best illustrated by possible descriptions of these two manifolds using a suitable scaling:
	for a given function $u \not = 0$ there is $t > 0$ such that $tu$ lies on the Nehari manifold and $u(\cdot/t)$ lies on the 
	Pohozaev manifold, in respective contexts.
We obtain, numerically, positive solutions for a semilinear problem and in particular, for the asymptotically linear case, which in turn was not tackled by previous algorithms in the literature - using the important fact proved by Pohozaev \cite{P} that any weak solution of an elliptic equation of type
\begin{equation}
\label{pri}
\left\{
\begin{array}{rllr}
\vspace{0.2cm}
&-\Delta u = g(u) \;\; \mbox{ in} \quad\Rn\;,\\

&u \in H^1(\Rn)\;,
\end{array}
\right.
\end{equation}
must satisfy the  Pohozaev identity \cite{P,Wi}
\begin{equation}\label{pohozaevid}
(N-2)\int_{\Rn}|\nabla u|^2dx= 2N \int_{\Rn}G(u)dx,
\end{equation}
where 
\begin{equation*}
G(u)=\displaystyle\int^u_0g(t)dt.
\end{equation*}

	We observe that, under very general hypotheses as in \cite{bl} the ground state solution is radially symmetric, therefore we are going to implement our algorithm in this setting of functions.
	
\section{Theoretical Background}

We consider the semilinear elliptic problem
\begin{equation}\label{prob}
\left\{
\begin{array}{rllll}
\vspace{0.2cm}
&-\Delta u+ \lambda u = f(u) \;\; \mbox{ in} \quad \Rn \;,\\

&u \in H^1(\Rn)\;,
\end{array}
\right.
\end{equation}
where $N\geq 3$ and $\lambda$ is a positive constant. Let $F(u)= \displaystyle\int^u_0f(t)dt$ and the associated functional to this problem defined in $H^1(\Rn)=W^{1,2}(\Rn)$ be
\begin{equation}\label{functional}
I(u)=\frac{1}{2}\displaystyle\int_{\Rn}(|\nabla u|^2+\lambda u^2)dx-\int_{\Rn} F(u)dx.
\end{equation}
Moreover, the functional (\ref{functional}) is well defined and $I\in C^1(H^1(\Rn), \mathbb{R})$ with
\begin{equation*}
I'(u)\varphi=\displaystyle \int_{\Rn}(\nabla u \nabla \varphi+ \lambda u \varphi) dx- \int_{\Rn}f(u)\varphi dx, \;\text{for all} \; \varphi \in H^1(\Rn).
\end{equation*}

Weak solutions $u$ of problem (\ref{prob}) are precisely the critical points of $I$, i.e, $I'(u)=0.$ We will assume that $f$ satisfies:
\begin{enumerate}
\item[$(f_1)$] $f\in C^1[0, +\infty)$;
\item[$(f_2)$] $f(u)=o(u)$ as $u\rightarrow 0$;
\item[$(f_3)$] There is a positive constant $a< \lambda$ such that
\begin{equation*}
\frac{f(u)}{u} \rightarrow a \;\;\;\; \text{or} \;\;\;\; \frac{f(u)}{u} \rightarrow + \infty, \;\;\text{as}\;\; u\rightarrow+ \infty; \;\; 
\end{equation*}
\item[$(f_4)$] There exist positive constants $a_1$ and $a_2$ such that
\begin{equation*}
|f'(u)|\leq a_1+a_2|u|^{p-2},
\end{equation*}
with $2<p<2^*:= \displaystyle \frac{2N}{N-2}$, if $N\geq 3$.
\end{enumerate}

\vspace{12pt}



Furthermore, we require that:

\begin{equation}
\left\{\begin{array}{rcll}\label{NQ}
\displaystyle \lim_{u\rightarrow +\infty}\left(\frac{1}{2}f(u)u-F(u)\right)&=& +\infty,&\\
\displaystyle \frac{1}{2}f(u)u-F(u)&>&0,& \;\; \forall\; u\in \mathbb{R}^+ \setminus \{0\}.
\end{array}\right.
\end{equation}

Without loss of generality, we may consider that $f \in C^1(\mathbb{R})$ is an odd function because we are focused on finding positive solutions. 

The first case of assumption $(f_3)$ implies that the problem is asymptotically linear at infinity and that the well-known Ambrosetti and Rabinowitz condition \cite{AR}
\begin{equation*}
	0<\theta F(u)\leq uf(u), \;\text{for some} \; \theta>2,
\end{equation*}
is not satisfied. We recall that any solution of (\ref{prob}) satisfies Pohozaev Identity  (\ref{pohozaevid}), in which

\begin{equation}\label{Gdef}
	G(u):=\displaystyle -\frac{\lambda}{2}u^2+F(u).
\end{equation}
We will further assume that: \\
 $(g_1)$ $\exists \; \xi>0$ such that $\displaystyle G(\xi)=\int^\xi_0g(s)ds>0.$

Let us define the Pohozaev manifold by
\begin{equation*}
\mathcal{P}=\{u\in H^1(\Rn)\setminus \{0\}:J(u)=0\},
\end{equation*}
with
\begin{equation}\label{j12}
J(u):=\displaystyle\int_{\mathbb{R}^N}|\nabla u|^2dx- \frac{2N}{N-2}\int_{\mathbb{R}^N} G(u)dx,
\end{equation}
and the constrained minimum of $I$ on $\mathcal{P}$ by
\begin{equation*}
m_{\mathcal{P}}:=\min_{u\in \mathcal{P}}I(u).
\end{equation*}

\begin{remark}\label{obs01}
Assumption $(g_1)$  implies that  $\mathcal{P}\neq \emptyset$. The proof of this fact is going to be postponed, subsequent to Lemma \ref{A1} (see Jeanjean and Tanaka, \cite{JT}).
\end{remark}
\vspace{0.2cm}

In the following, we will use the notation
$$
\left\langle u,v\right\rangle :=\int_{\mathbb{R}^{N}}\left(  \nabla
u\cdot\nabla v+ uv\right) \,dx  \text{\quad} \text{and}  \text{\quad}\Vert u\Vert^{2}%
:=\int_{\mathbb{R}^{N}}\left(  |\nabla u|^{2}+ u^{2}\right) \,dx 
$$
for the inner product and norm in the Hilbert space
$H^{1}(\mathbb{R}^{N})$, respectively.

We recall that $I$ in (\ref{functional}) satisfies the Palais-Smale condition \cite{palais} at level $c \in \mathbb{R}$ ((PS)$_c$ for short) if any sequence $\{u_n\} \subset H^1(\Rn)$ such that $I(u_n) \to c$ and $\|I'(u_n)\|_{{H^{-1}}} \to 0$ contains a convergent subsequence, where $H^{-1}$ is the dual space of $H^{1}(\mathbb{R}^{N})$. Furthermore, let us review the  Mountain Pass Theorem of Ambrosetti and Rabinowitz \cite{AR}.\\

\begin{theorem}\label{mountainpass}
Assume $I\in C^1 (H^1(\Rn),\mathbb{R})$ such that, $I(0)=0$ and\\

$(I_1)$ there exist constants $\rho, \alpha>0$ such that $I|_{\partial B_{\rho}(0)}\geq\alpha,$ and\\

$(I_2)$ there exists an $e\in H^1(\Rn) \setminus \overline {B_{\rho}(0)}$ and $I(e)\leq 0$.
Define
\begin{equation}\label {defc}
c:=\inf_{\gamma \in \Gamma}\max_{u\in\gamma([0,1])}I(u),
\end{equation} 
where
\begin{equation*}
\Gamma=\{\gamma\in C([0,1], H^1(\Rn))| \gamma(0)=0, \gamma(1)=e\}.
\end{equation*}
Then, if $I$ satisfies $(PS)_c$, the level $c$ is a critical level of $I$, i.e, there exists $u\in H^1(\Rn)$ such that $I(u)=c$ and $I'(u)=0.$
\end{theorem}

We recall that a solution $v$ of (\ref{prob}) is said to be a least energy solution if, and only if
\begin{equation}\label{mdef}
I(v)=m,\;\; \text{where} \;\;m:= \inf \{I(u); u\in H^1(\mathbb{R}^N)\setminus \{0\}\; \text{is a solution of}\; (\ref{prob})\}.
\end{equation}

\begin{remark}
Since any solution of (\ref{prob}) satisfies Pohozaev identity, then $m_{\mathcal{P}}\leq m$. We will show in Lemma \ref{lem4} that in fact $m=m_{\mathcal{P}}$. 
\end{remark}

The important work of Jeanjean and Tanaka \cite{JT} was decisive to relate $m$ and $c$ in a theorem which states:\\
\begin{theorem} \label{theoremjeanjean}
Assume $(f_1)-(f_4)$ and $(g_1)$. Then $m=c$ holds, where $m,\; c>0$ are defined in (\ref{mdef}) and (\ref{defc}), respectively. That is, the mountain pass value gives the least energy level. Morever, for any least energy solution $v$ of (\ref{prob}), there exists a path $\gamma\in \Gamma$ such that $v\in \gamma([0,1])$ and
$$\max_{\tau\in[0,1]}I(\gamma(\tau))=I(v).$$
\end{theorem}

Note that, from the assumptions $(f_1)$, $(f_2)$ and $(f_4)$, given $\varepsilon >0$, there exists $C_1=C_1(\varepsilon)>0$, such that
\begin{equation}\label{Fcx} 
|F(u)|\leq \frac{\varepsilon}{2}u^2+C_1|u|^p, \;\;\;\; 2<p<2^*.
\end{equation}

The following lemmas describe the analytical tools which are going to support the construction of the new algorithm.\\

\begin{lemma}\label{c1}
Let the functional $J:H^1(\Rn)\rightarrow \mathbb{R}$ be defined as in $(\ref{j12})$.
Then\\
$(1) $ there exists $\sigma>0$ such that $\left\|u\right\|>\sigma$, for all $u \in \mathcal{P}$;\\
$(2)\; \mathcal{P}=\{u\in H^1(\Rn) \setminus \{0\}|\;J(u)=0\}$ is closed;\\
$(3)\; \mathcal{P}$ is a manifold of class $C^1$.\\
\end{lemma}

\begin{proof}
	Verification of (1). Since $u\in \mathcal{P}$, we have:
	\begin{equation*}
	\int_{\mathbb{R}^N}|\nabla u|^2\,dx=2^*\int_{\mathbb{R}^N} G(u)\,dx,
	\end{equation*}
	thus,
	\begin{equation*}
	\int_{\mathbb{R}^N}\left(|\nabla u|^2 + \frac{\lambda N}{N-2}u^2\right) dx=2^* \int_{\mathbb{R}^N} F(u)\,dx.
	\end{equation*}
	Hence, there is a constant $M$, given by $\displaystyle M:=\min\left\{1,\frac{\lambda N}{N-2}\right\}$, such that
	\begin{equation*}
	M\left\|u\right\|^2\leq 2^*\int_{\mathbb{R}^N}  F(u)\,dx,
	\end{equation*}
	and, using $(\ref{Fcx})$, it follows that
	\begin{equation*}
	M\left\|u\right\|^2\leq 2^*\int_{\mathbb{R}^N}\left(\frac{\varepsilon}{2}|u|^2+C_1|u|^p\right) dx.
	\end{equation*}
	Now, taking $\varepsilon>0$ such that 
	$2^*\varepsilon<M$,
	we obtain
	\begin{equation*}
	\frac{M}{2}\left\|u\right\|^2\leq2^* C_1\int_{\mathbb{R}^N}|u|^p \,dx, \;\;2<p<2^*.
	\end{equation*}
	Therefore, by the continuous Sobolev embedding $H^1(\mathbb{R}^N) \hookrightarrow L^p(\mathbb{R}^N)$,  there is $\sigma>0$ such that
	$\label{limita}
	\sigma\leq\left\|u\right\|^{p-2}.
	$
	\\
Verification  of (2) and (3). By the definition of  $J$, we have
\begin{equation*}
J(u)=\int_{\mathbb{R}^N}|\nabla u|^2\,dx-2^*\int_{\mathbb{R}^N} G(u) \,dx,
\end{equation*}
which is a functional of class $C^1$. Thus
\begin{equation*}
\mathcal{P}=\{J(u)=0, u\not = 0\}= J^{-1}(\{0\}) \setminus \{0\},
\end{equation*}
and it follows that $\mathcal{P}$ is a  closed set, since by (1), $0 \in H^1({\mathbb{R^N}})$ is an isolated point of the level set $J=0$. Moreover, using $(\ref{NQ})$, we have
\begin{eqnarray*}
\frac{1}{2}J'(u)u&=&2^*\int_{\mathbb{R}^N}\left(G(u)-\frac{1}{2}g(u)u\right) dx\\
&=&2^*\int_{\mathbb{R}^N}\left(F(u) -\frac{1}{2}f(u)u\right) dx<0.
\end{eqnarray*}
Therefore, $J'(u)\neq0$ and, thus, $\mathcal{P}$ is a manifold of class $C^1$ in $H^1(\Rn)$.\\
 \hfill \end{proof}

\begin{lemma}\label{A1}
For each $u\in {H^1(\Rn)}\backslash\{0\}$ with $\displaystyle \int_{\mathbb{R}^N} G(u)>0$ there exists a unique real number $t_0>0$ such that $u(\frac{\cdot}{t_0})\in \mathcal{P}$ and $I(u(\frac{\cdot}{t_0}))$ is the maximum of the function $t\mapsto I(u(\frac{\cdot}{t})), \; t>0.$
\end{lemma}
\begin{proof}
Consider the following function $h$, given by
\begin{equation*}
h(t):=I\left(u\left(\frac{.}{t}\right)\right)=\frac{t^{N-2}}{2}\int_{\mathbb{R}^N}|\nabla u|^2+\frac{\lambda t^N}{2}\int_{\mathbb{R}^N}u^2-t^N \int_{\mathbb{R}^N}F(u).
\end{equation*}
Thus, for $N\geq 3$:
\begin{equation*}
h'(t)=\frac{(N-2)t^{N-3}}{2}\int_{\mathbb{R}^N}|\nabla u|^2+Nt^{N-1}\int_{\mathbb{R}^N}\left[\frac{\lambda}{2}u^2- F(u)\right],
\end{equation*}
and $h'(t)=0$ if, and only if,
\begin{equation*}
t^{N-3}\left(\frac{(N-2)}{2}\int_{\mathbb{R}^N}|\nabla u|^2+Nt^2\int_{\mathbb{R}^N}\left[\frac{\lambda }{2} u^2- F(u)\right]\right)=0.
\end{equation*}
Therefore we have either $t=0$ or
\begin{equation}\label{t_proj}
t^2=\frac{(N-2)\displaystyle\int_{\mathbb{R}^N}|\nabla u|^2}{2N\displaystyle\int_{\mathbb{R}^N}\left[-\frac{\lambda }{2}u^2+F(u)\right]}=\displaystyle\frac{(N-2)\displaystyle\int_{\mathbb{R}^N}|\nabla u|^2}{2N\displaystyle\int_{\mathbb{R}^N}G(u)}\cdot
\end{equation}
\hfill{}
\end{proof}

\begin{remark}
The statement of Lemma \ref{A1} is also true for $N=2$ and can be found in Jeanjean and Tanaka \cite{JT}. 
\end{remark}
Now, we can prove the claim in Remark \ref{obs01}. In fact, first note that, for $G(u)$ as in (\ref{Gdef}), $G(0)=0$ and consider the ball $B_1(0)$ and $A_{\varepsilon}=B_{1+\varepsilon}(0)\setminus B_1(0)$, for a given $0<\varepsilon<1$.  Moreover, we define
\begin{equation*} 
u_\xi(x)=
\left\{\begin{array}{rl}
\xi& \text{if} \;\; x\in \;B_1(0)\;,\\
0& \text{if} \;\; x\in \; B^c_{1+\varepsilon}(0)\;,
\end{array}\right.
\end{equation*}
so that $u_\xi(x)=u_\xi(|x|)$ is a continuous, non-increasing function of $r=|x|$. Since $0<u_\xi(x)<\xi$, then $|G(u_\xi(x))|<C$. In addition,
\begin{equation}\label{EqGU}
\int_{\mathbb{R}^N}G(u_\xi)dx=\int_{B_{1+\varepsilon}(0)}G(u_\xi)dx=\int_{B_1(0)}G(u_\xi)dx + \int_{A_{\varepsilon}}G(u_\xi)dx,
\end{equation}
\begin{equation}\label{Eqg1}
\int_{B_1(0)}G(u_\xi)dx= G(\xi) \, \text{meas}(B_1(0))>0 
\end{equation}
and
\begin{equation}\label{Eqg2}
\Big|\int_{A_{\varepsilon}}G(u_\xi)dx\Big|\leq \int_{A_{\varepsilon}}|G(u_\xi)|dx \leq C \, \text{meas} \, (A_{\varepsilon})=C\varepsilon,
\end{equation}
where $\text{meas}(A)$ denotes the Lebesgue measure of the set $A$.

Therefore, taking $\varepsilon$ sufficiently small and applying $(\ref{Eqg1})$, $(\ref{Eqg2})$ in $(\ref{EqGU})$, we obtain that
\begin{equation}\label{EqGU1}
\int_{\mathbb{R}^N}G(u_\xi)dx=\int_{B_2(0)}G(u_\xi)dx\geq G(\xi)\, \text{meas} \,(B_1(0))-C\varepsilon> 0.
\end{equation}
By Lemma \ref{A1}, there exists $t_{\xi}$  such that $u_{\xi}(\frac{\cdot}{t_{\xi}})\in \mathcal{P}.$ 
The next lemma shows that $\mathcal{P}$ is a natural constraint for the functional $I$.

\begin{lemma}\label{lem4}
A function $u\in H^1(\Rn)\setminus \{0\}$ is a critical point of $I$ if and only if $u$ is a critical point of $I$ restricted to $\mathcal{P}$.
\end{lemma}

\begin{proof}
The proof follows \cite{JS}. Let $u$ be a critical point of the functional $I$, restricted to $\mathcal{P}$.  By the Lagrange Multiplier Theorem for Banach spaces (see Theorem 26.1, \cite{Deim}), we have
$$I'(u) + \eta J'(u) =0,\;\; \text{for \;some} \;\;\eta\in \mathbb{R}.$$
Let us show that $\eta= 0$. Applying $u$ in the above equation, we obtain:
\begin{equation}\label{arueira}
I'(u)u + \eta J'(u)u = 0,	
\end{equation}
and since $N \geq 3$ it is equivalent to
\begin{eqnarray*}
0 &=&\Intrn|\nabla u|^2 + \lambda u^2dx - \Intrn f(u)u \;dx \\
&&+ \eta\left((N-2)\Intrn|\nabla u|^2 - N\Intrn f(u)u-\lambda u^2 \;dx\right).
\end{eqnarray*}
This yields the following equation
\begin{equation*}
-\Delta u + \lambda u - f(u) + \eta\left(-(N-2)\Delta u + \lambda Nu - Nf(u)\right) =0,
\end{equation*}
which, in turn, can be rewritten as
\begin{equation}
-(1+\eta(N-2))\Delta u + \lambda(1+\eta N)u = (1+\eta N)f(u). \label{restricao}
\end{equation}\\
This equation has the Pohozaev manifold associated with it, given by ${\mathcal{H}}^{-1}(\left\lbrace 0\right\rbrace)$, where
\begin{equation*}
\mathcal{H}(u) := \Frac{(1+\eta(N-2))(N-2)}{2}\Intrn |\nabla u|^2 \,dx - N\Intrn \overline{G(u)}\, dx,
\end{equation*}
with
\begin{equation*}
\overline{G(u)}:= (1+\eta N)F(u) - \lambda\Frac{(1+\eta N)}{2}u^2.
\end{equation*}
Thus, $\mathcal{H}$ can be rewritten as
\begin{eqnarray}
\nonumber \mathcal{H}(u)&=& \Frac{(1+\eta(N-2))(N-2)}{2}\Intrn |\nabla u|^2dx \\
\nonumber && -  N\Intrn \left((1+\eta N)F(u)-\lambda\Frac{(1+\eta N)}{2}u^2\right)dx\\
\nonumber & = & \Frac{(1+\eta(N-2))(N-2)}{2}\Intrn |\nabla u|^2 \;dx \\
\label{lagrangem} && -  N(1+\eta N)\Intrn\left(  F(u) - \lambda\Frac{u^2}{2}\right)dx.
\end{eqnarray}
However, since $u\in \mathcal{P}$, then $J(u)=0$, and thus, by (\ref{lagrangem}), 
\begin{eqnarray*}
\mathcal{H}(u) =-\eta(N-2)\Intrn |\nabla u|^2 \;dx \;.
\end{eqnarray*}
Even further, $u$ is a solution of the equation (\ref{restricao}) and therefore satisfies $\mathcal{H}(u) = 0$. Thus, we obtain
\begin{equation*}
-\eta(N-2)\Intrn |\nabla u|^2 \;dx = 0\;.
\end{equation*}

Since $N\geq3$ and $\Intrn |\nabla u|^2dx>0$, we have $\eta =0$. Thus equation (\ref{arueira}) is actually $I'(u)=0$ and $u$ is a critical point of $I$.
\hfill{}
\end{proof}
\begin{lemma}\label{lemcoercive}
Let $v$ and $w$ in $H^1(\Rn)$, 
$\Phi(\alpha):=I(w+\alpha v)$,
 $\gamma(\alpha):= 
(w + \alpha v)(\frac{\cdot}{t(\alpha)} ) \in \mathcal{P}$
 and $\Psi(\alpha):= I(\gamma(\alpha))$.
If $t(\alpha)$ is bounded from below by a positive constant,
then $\lim\limits_{\alpha \to \pm \infty} I(\gamma(\alpha))=
+\infty$
and $\min I(\gamma(\alpha))= I(\gamma(\hat{\alpha}))$ is attained for some $\hat{\alpha} \in \R$.
Otherwise, if $\lim_{\alpha_j \to + \infty} t(\alpha_j)=0$ on a subsequence $(\alpha_j)_{j \in \mathbb{N}}$
and there exists $\delta > 0$ such that 
$\Phi^\prime (\alpha) <0$, for $0 < \alpha < \delta$, 
then either there is $ \hat{\alpha} >0$ which is a point of local minimum of $I(\gamma(\alpha))$ or
$I(w + \alpha v) \leq I((w + \alpha v)(\frac{\cdot}{t(\alpha)})) <I(w)$ for $\alpha >0$.
\end{lemma}
\begin{proof}
It holds that
\begin{eqnarray*}
J(u)&=&\displaystyle\int_{\mathbb{R}^N}|\nabla u|^2 \,dx- \frac{2N}{N-2}\int_{\mathbb{R}^N} G(u)\,dx\\
&=&
2^* \left(I(u)- \frac{1}{N} \left\|\nabla u\right\|_2^2 \right).
\end{eqnarray*}
If $u \in \mathcal{P}$, then
\begin{equation}
\left\|\nabla u\right\|_2^2=NI(u).
\label{coercive}
\end{equation}
Putting  $u = \gamma(\alpha)$, we have two possibilities, either $t_\alpha \geq \bar t >0$, for some positive constant $\bar t$, hence $\left\|\nabla (w + \alpha v)(\frac{\cdot}{t_\alpha}) \right\|_2^2
=t_\alpha^{N-2}||\nabla (w+\alpha v)||_2^2
  \rightarrow + \infty$ as $\alpha \rightarrow \pm \infty$, and so, by (\ref{coercive}) it follows that $\lim\limits_{\alpha \to \pm \infty} I(\gamma(\alpha))=+\infty$. The minimum is attained because $I$ and $\gamma(\alpha)$ are continuous.
  
Otherwise, up to a subsequence, $t_\alpha \to 0$ as $\alpha \rightarrow + \infty$ or $\alpha \rightarrow - \infty$. By assumption the first case holds and also 
$\Phi^\prime (\alpha) <0$, so that
$I(w+\alpha v)< I(w)$, for $0 < \alpha < \delta$. 
Since $I \in C^1(H^1(\Rn), \mathbb{R})$, $\gamma \in C^1 (\mathbb{R})$ and
$\Phi^\prime(\alpha)=I'(w+\alpha v)v <0$, then $\Psi^\prime(\alpha)=I'(\gamma (\alpha))\gamma^\prime(\alpha)< 0$, for $\alpha$ positive and sufficiently small, because $||(w+ \alpha v) - \gamma (\alpha)|| \to 0$ as $\alpha \to 0$.
If $\Psi^\prime $ changes sign, then there is $\hat \alpha$ such that
$I(\gamma(\hat \alpha))$ is a local minimum. However, if
$\Psi^\prime (\alpha)$ does not change sign for $\alpha >0$, then
by Lemma \ref{A1}, $\Psi^\prime (\alpha)<0$ and $\Psi(0)= I(w)$ we obtain
\begin{equation}\label{secondcase}
I(w+\alpha v) \leq I(w + \alpha v(\frac{\cdot}{t_\alpha})) < I(w).
\end{equation}
\hfill{}
\end{proof}
\begin{remark}
	Note that in case $t_\alpha \to 0$ as $\alpha \rightarrow - \infty$, one may repeat the previous proof exchanging $v$ for $-v$ and $\alpha$ for $-\alpha$.
	\end{remark}			
\section{The steepest descent direction}
\label{sec_def_stepest}

The steepest descent direction  at $ w_1 \in H^1(\mathbb{R}^N)$   corresponds to finding $\hat v \in H^1(\mathbb{R}^N)$ with $\|\hat v\|=1$ such that 

\begin{equation}\label{steepestdef}
	\frac{I(w_1 + \varepsilon \hat v) - I (w_1)}{\varepsilon}
\end{equation}
is as negative as possible as $\varepsilon \to 0$. This is equivalent to finding the minimum of the Fr\'{e}chet  derivative at $w_1$ on $\hat v$,
i.e. $I'(w_1)\hat v$, subject to the constraint $\|\hat v\| =1$.

The device for solving numerically the steepest descent direction can be found by means of a linear equation detailed by J. Hor\'{a}k in \cite{JH}. For the sake of completeness we recall it here. Introducing the Lagrange Multiplier $\mu$, we therefore look for the unconstrained minimum of the functional $L: H^1(\mathbb{R}^N) \to \mathbb{R}$, defined by 
$$L(\hat v):= I'(w_1)\hat v + \mu \int_{\mathbb{R}^N} \nabla \hat v \cdot \nabla \hat v + \hat v \cdot  \hat v \;dx,$$
or, equivalently,
\[
L(\hat v):= \int_{\mathbb{R}^N} \nabla w_1 \cdot \nabla \hat v + \lambda w_1 \hat v - f(w_1)\hat v + \mu (|\nabla \hat v|^2 + | \hat v|^2 )\; dx.
\]\\
The Fr\'{e}chet derivative of $L$ exists and is given by
\[
L'(\hat v)\phi= \int_{\mathbb{R}^N} \nabla w_1 \cdot \nabla \phi + \lambda w_1 \phi - f(w_1)\phi + 2 \mu (\nabla \hat v \cdot \nabla \phi + \hat v \phi)\; dx,
\]
for any  $\phi \in H^1(\mathbb{R}^N)$. Hence, $L'(\hat v)= 0$ corresponds to a weak solution $\hat v \in H^1(\mathbb{R}^N)$ of the linear equation:

\begin{equation}\label{steepest}
2 \mu (\Delta \hat v - \hat v)=-\Delta w_1 + \lambda w_1 - f(w_1)\;.	
\end{equation}

\section{The Mini-Max Algorithm using Pohozaev (MMAP)}\mbox{}
\label{algorithm}
\vspace{0.2cm}
\\
The general approach in solving numerically the proposed problem is the following: we restate the problem in a variational formulation on a Hilbert space with a constraint that defines the Pohozaev manifold $\mathcal{P}$ and use the steepest descent method allied with projections on $\mathcal{P}$ to find minima of the functional $I$ constrained to the direction found by the former. By iterating such a process, we arrive at the minimum of $I$ constrained to $\mathcal{P}$, which is the ground state solution obtained by the Mountain Pass theorem. The formulated algorithm is derived from the rigorous theoretical results aforementioned and it converges to the positive ground state solution of problem (\ref{prob}). The main idea is to descend along paths projected on the Pohozaev manifold $\mathcal{P}$, which are precisely $\gamma(t) = u(\frac{\cdot}{t})$ constructed from Lemma \ref{A1}. 

In the work of Choi and McKenna \cite{YP}, a constructive form of the Mountain Pass Theorem, first formulated by Aubin and Ekeland \cite{AE}, was implemented numerically by allying the finite element method with a method of steepest descent. This was done by starting with a local minimum and connecting it with a path to a point $e$ with $I(e) \le 0$ of lower altitude (Theorem \ref{mountainpass}), finding the maximum of $I$ along this path, then deforming it in such a way as to make the maximum along the path decrease as fast as possible and, finally, if that maximum turns out to have been a critical point, they stop, or else, repeat this process. They apply the algorithm in a rectangle to a homogeneous superlinear nonlinearity of type $u^p$, $1<p<2^*$, 
but this
algorithm has been applied to problems with no symmetry assumptions, even on unbounded domains.

A couple of years after \cite{AE}, Ding and Ni \cite{Ding} showed that a solution to the more general problem (\ref{pri}) with the nonlinearity obeying a monotonicity condition on a bounded domain (but also in $\mathbb{R^N}$) exists by a constrained minimization argument on the so called Nehari manifold. Then, Chen, Ni and Zhou \cite{GJW}  used this approach to adapt the preceding algorithms and solve for more general bounded domains with projections on Nehari manifold. However, the limitation of this idea is that a unique projection is required in order to apply the constrained minimization problem successfully which, in turn, depends on the monotonicity assumption. We weaken this condition by, rather than constraining the problem to the Nehari manifold, following with the clever idea of projections on the Pohozaev manifold since, for (\ref*{pri}), they are always guaranteed to be unique. This extends our framework to more general problems, including non-homogeneous superlinear problems in unbounded domains, as well as asymptotically linear problems, homogeneous or not. 

Besides \cite{GJW}, as already mentioned, the numerical minimization under a general constraint, its application to the Nehari manifold and its relation to the Mountain Pass Algorithm were studied in \cite{JH}. Concerning the actual constraint of the Pohozaev manifold, our algorithm is initialized in a different manner: while in \cite{JH} one would need to generate a discretized path connecting two points $e_1$, $e_2$ in the Pohozaev manifold $\mathcal{P}$ (taken as local minima of the associated functional, numerically found using a constrained steepest descent method), where this path would be represented numerically by a collection of finite points, and so find the maximum of $I$ along such a path, our algorithm takes a function $w \in H^1(\Rn)$ such that $\int_{\Rn} G(w) > 0$ and finds the maximum of the associated functional restricted to $\mathcal{P}$ by means of a direct formula using the parameter $t$ in (\ref{t_proj}),  and hence this approach is expected to lighten the computational cost at this step.

In our work, the proposed algorithm can, in fact, be understood as the constrained steepest descent method of \cite{JH}, applied here to the Pohozaev manifold, with one main difference: in \cite{JH} the orthogonal projection of the gradient to the tangent space to the manifold is used, whereas here a different kind of projection is employed (Lemma \ref{lemcoercive}), as we always reproject on $\mathcal{P}$ as we descend along the steepest descent direction. Furthermore, since by Theorem \ref{theoremjeanjean} the ground state solution corresponds to the minimum on the Pohozaev manifold, our algorithm should be fitter in finding this minimal action solution.

The general idea of the new algorithm is made clear in the sequel:\\

{\bf Step 1.} Take an initial guess $ w_0 \in H^1(\mathbb{R}^N)$ such that  $w_0 \not = 0$  and $ \int_{\Rn} G(w_0 ) > 0$, under the assumption that $0$ is a local minimum of $I$, since $I$ has the Mountain Pass geometry;\\

{\bf Step 2.}  Find $t_*>0$ by (\ref{t_proj}) such that
\begin{equation}
  I \left(w_0\left(\displaystyle{\frac{.}{t_*}}\right)\right) = \max_{t>0}I\left(w_0\left(\displaystyle{\frac{.}{t}}\right)\right),
  \end{equation}
and set $w_1 = w_0\left(\displaystyle{\frac{.}{t_*}}\right)$. This is possible because $\int_{\Rn} G(w_0) > 0$ and hence one can use Lemma \ref{A1};\\

{\bf Step 3.}  Find the steepest descent direction $v \in H^1(\mathbb{R}^N)$ at $w_1 \in H^1(\mathbb{R}^N)$, from (\ref{steepest}), obtaining $v =- \nabla I (w_1 ) $.  If $\|v\| < \varepsilon$, then output and stop. Else, calculate $\hat v = v/2\mu$, such that $\mu = ||v||/2$ and $||\hat v|| = 1$, and then go to the next step.\\

{\bf Step 4.} 
For $0<\alpha_0$ small, there exists $t(\alpha_0)$ such that $(w_1 + \alpha_0 \hat v) \left (\frac{\cdot}{t(\alpha_0)} \right) \in \mathcal{P}$. Fix $K \in \mathbb{N}$ and iterate $\alpha_k := k \alpha_0$, for $k \in \mathbb{Z} $, $1\leq k \leq K$ and $(w_1 + \alpha_k \hat v) \left (\frac{\cdot}{t(\alpha_k)} \right) \in \mathcal{P}$. 
In view of Lemma \ref{lemcoercive}, we can either find $\hat \alpha$ such that $$I \left ((w_1 + \hat \alpha \hat v) \left (\frac{\cdot}{t(\hat \alpha)} \right) \right) = \min \limits_{\alpha_k} I \left ((w_1 +  \alpha_k \hat v) \left (\frac{\cdot}{t( \alpha_k)} \right) \right) $$ or such a minimum is not attained and $$
I(w_1+\alpha \hat v) \leq I(w_1 + \alpha \hat v(\frac{\cdot}{t_\alpha})) < I(w_1).$$ In the former case, proceed to Step 5. In the latter, let $1 \leq k_0 \leq K$ be the largest $k$ such that $ \int_{\Rn} G(w_1+k_0 \alpha_0 \hat v ) > 0$, return to Step 1 and consider a new initial guess $w_0:= w_1+ k_0\alpha_0 \hat v$. 
\\

{\bf Step 5.} Redefine
$w_0:= w_1 + \hat \alpha \hat v$.
 Go to Step 2. \\

\section{Numerical implementation of the algorithm}

The algorithm presented in the previous section is applicable for general nonlinearities, which satisfy the hypotheses stated in the introduction and can be applied to problems with no symmetry assumptions, provided one works in a scenario to regain compactness in $\mathbb{R^N}$. However, for the sake of simplicity, we are going to implement for nonlinearities which satisfy conditions that imply that the ground state solution is radially symmetric.

\subsection{Radial symmetry}
Since $f\in\mathcal{C}^{1}(\mathbb{R})$ is odd and $f$ satisfies $(f1)-(f4)$, a classical
result of Berestycki and Lions \cite{bl} establishes the existence of a ground state
solution $\omega\in\mathcal{C}^{2}(\mathbb{R}^{N})$ to the problem
(\ref{prob}), which is positive, radially symmetric and decreasing in the
radial direction (see Theorem 1 in \cite{bl}). In fact, by Li and Ni, if $g'(0) \le 0$ then any positive solution of (\ref{pri}) is, up to a translation, radially symmetric (see Theorem 1 in \cite{lini}). Moreover, this radial positive solution is unique when extra hypotheses are satisfied (see Serrin and Tang \cite{serrin_tang}).
Therefore, we are going to restrict ourselves to the $H_{rad}^1(\mathbb{R}^N)$, the subspace of radial functions of $H^1(\mathbb{R}^N)$, without loss of generality. 
Since the functions are all radially symmetric, the integrals are calculated in the real line by a change from cartesian to spherical variables, with $u(r,\theta, \phi) = u(r)$.
Moreover, all the partial differential equations involved are transformed into ordinary differential equations in the radius variable. Since we are working on $\mathbb{R}^3$, the problem is reduced to:

\begin{equation}
\begin{cases}
 -u''(r) -\dfrac{2}{r} u'(r) + \lambda u(r) = f(u(r)), \; r>0, \\
\vspace{0.1cm}
\;\;\; u(r) \rightarrow 0, \; r \rightarrow +\infty,\\
\vspace{0.1cm}
\;\;\; u'(0) = 0.
\end{cases}
\label{problem_radial}
\end{equation}
Moreover, our functional $I$, projected on $\mathcal{P}$ depends on:
\begin{equation*}
h(t):=I\left(u\left(\frac{.}{t}\right)\right)=4\pi\left( \frac{t^{2}}{2}\int_{0}^{+\infty}|u'|^2 r^2 dr+\frac{\lambda t^3}{2}\int_{0}^{+\infty}|u|^2 r^2 dr-t^3 \int_{0}^{+\infty}F(u) r^2 dr \right)
\end{equation*}
and its derivative is given by:
\begin{equation*}
h'(t):=I'\left(u\left(\frac{.}{t}\right)\right)=4\pi\left(t\int_{0}^{+\infty}|u'|^2 r^2 dr+3\frac{\lambda t^2}{2}\int_{0}^{+\infty}|u|^2 r^2 dr-3t^2 \int_{0}^{+\infty}F(u) r^2 dr \right).
\end{equation*}
Therefore, the value of $t$ that projects $u$ on $\mathcal{P}$ is directly given by $h'(t) = 0$:
\begin{equation}\label{t_proj2}
t^2=\frac{\displaystyle\int_{0}^{+\infty}|u'|^2 r^2 dr}{3\displaystyle\int_{0}^{+\infty}\left[-\frac{\lambda }{2}|u|^2+F(u)\right] r^2 dr}=\displaystyle\frac{\displaystyle\int_{0}^{+\infty}|u'|^2 r^2 dr}{3\displaystyle\int_{0}^{+\infty}G(u) r^2 dr}\cdot
\end{equation}

\subsection{Discretisation and numerical methods}

We start by noting that the algorithm presented in Section \ref{algorithm} does not involve solving directly (\ref{problem_radial}) and, therefore, it does not need to be discretised or treated numerically otherwise. The parts of the algorithm that need to be treated numerically are the calculations of the functional $I(u)$ and of the projection parameter $t$, which involve the calculation of integrals, and the calculation of the steepest descent direction, which is given by the Poisson problem in equation (\ref{steepest}). We will describe briefly below how these were implemented.

The integrals involved in the MMAP algorithm were evaluated using a standard trapezoidal rule,
\begin{equation}
\int_a^bh(r) dr = \left(\frac{h(a)+h(b)}{2}+\sum_{i=1}^{M-1}h(r_i)\right)\Delta r + \mathcal{O}(\Delta r^3),
\end{equation}
for  a function $h(r)$, where $\Delta r = 1/M$ is the space step taken to discretise the interval $[a,b]$ in which the integral is defined. Note that the truncation error in this approximation is $\mathcal{O}(\Delta r^3)$.

The steepest descent direction, given by the solution of (\ref{steepest}), can also be written in terms of a radially symmetric problem, that is:
\begin{equation}
  +v''(r) +\frac{2}{r} v'(r) - \lambda v=   -w_1''(r) -\frac{2}{r} w_1'(r) + \lambda w_1 - f(w_1),
  \label{steepest_radial}
\end{equation}
with $w_1$ given from Step 2, and with boundary conditions given by
\begin{equation}
  v(r) \rightarrow 0, \; r \rightarrow +\infty, ~~\textrm{and}~~ v'(0) = 0.
  \label{bc_steepest_radial}
\end{equation}
We use second order centered finite differences to discretise (\ref{steepest_radial}). Defining $v_i = v(r_i)$, and similarly with $w_1$, we obtain the discretised version of (\ref{steepest_radial}) as:
\begin{equation}
  \alpha v_{i+1} + \beta v_{i} + \gamma v_{i-1} = \alpha' {w_1}_{i+1} + \beta' {w_1}_{i} + \gamma' {w_1}_{i-1} + f({w_1}_{i}),
  \label{steepest_discrete}
  \end{equation}
with
\begin{equation}
  \alpha = \displaystyle\frac{1}{\Delta r ^2} + \frac{1}{r_i \Delta r}, ~\beta = -\displaystyle\left( \frac{2}{\Delta r^2} + 1\right), ~\gamma = \displaystyle\frac{1}{\Delta r ^2} - \frac{1}{r_i \Delta r}
\end{equation}
  and
\begin{equation}
  \alpha' = -\alpha,~\beta' =\displaystyle \frac{2}{\Delta r^2} + \lambda, ~\gamma'= - \gamma.
  \end{equation}

We now observe that (\ref{steepest_discrete}) is a linear system of $M+1$ equations in terms of $v_i$, which is solved by an SOR method with relaxation parameter chosen as $tol_{SOR} = 1.9$.

Note that the boundary conditions of (\ref{steepest_radial}), given in (\ref{bc_steepest_radial}), also have to be discretised. The first boundary condition in (\ref{bc_steepest_radial}) is taken to be $v_M=0$, where $v_M = v(R^*)$, with $R^*$ large enough so that this approximation is adequate. We discuss the influence of the choices of $R^*$ in Subsection \ref{size_domain}. The second boundary condition in (\ref{bc_steepest_radial}) is discretised using a second order forward finite difference, which gives $v_0 = \displaystyle \frac {4v_1-v_2}{3}$.

Finally, we note from Section \ref{sec_def_stepest} that the steepest descent function has to be normalised, and therefore, the solution obtained in (\ref{steepest_discrete}) has to be divided by ${2\mu}$, as discussed in Step 3, so that we can control with $\alpha_k$ how much we descend along the steepest descent direction.  However, we must keep track of the actual value of the norm of the steepest descent function found, since we need it to assess the convergence of the algorithm, as stated on the Step 3 in Section \ref{algorithm}.

\subsection{Pohozaev projection step}

Given an initial guess $w_0 \in H_{rad}^1$, one can verify that $\int_{\Rn}G(w_0) > 0$, which is done by calculating this integral using the trapezoidal rule on our mesh $\Omega$. Then, by Lemma \ref{A1}, we calculate $t_*$ by solving for $t$ in (\ref{t_proj2}). In the process of setting $w_1 = w_0\left(\displaystyle{\frac{.}{t_*}}\right)$, the points $r_i$ in $\Omega$ may not be appropriate for $w_1$ because the rescaled points $\dfrac{r_i}{t_*}$ may not be in the mesh $\Omega$. In order to avoid having to interpolate the function $w_1$ to obtain its projection on $\mathcal{P}$, we actually rescale the interval $\Omega$ by taking $r_i \rightarrow t^* r_i$ so that we find the new $r_i$-coordinates for the values of $w_1$ that we already have calculated on the mesh.

Moreover, on Step 4, projections of the line $w_1 + \alpha \hat v$ with varying $\alpha$ are calculated for $t(\alpha)$ by, again, solving for $t$ in (\ref{t_proj2}). Note that this is done in the same setting as Step 2. When evaluating the projection $w_1 + \alpha \hat v$ in $I$, the level of the functional decreases until it reaches a minimum, which is guaranteed by Lemma \ref{lemcoercive}. This will be further explained in the next section.

\subsection{Descending on Pohozaev manifold}
 First, evaluate the functional $I$ on $w_1 \in \mathcal{P}$. We consider a given $\alpha_0$ (typically we choose $\alpha_0 = 10^{-1}$) and we evaluate $I\left((w_1 + \alpha_k \hat v)  \left(\dfrac{\cdot}{t( \alpha_k)}\right)\right)$, for increasing integers $k$, until we find $k = \bar k$ such that 
\begin{equation}\label{descent_num}
  I\left((w_1 + \alpha_{\bar k} \hat v)  \left(\frac{\cdot}{t( \alpha_{\bar k})}\right)\right) > I\left(w_1 + \alpha_{\bar k - 1} \hat v)  \left(\frac{\cdot}{t( \alpha_{\bar k - 1})}\right)\right).
\end{equation}
When this $\bar k$ is found, we redefine $w_1^{new} := w_1 + \alpha_{\bar k - 1} \hat v$ and project it on $\mathcal{P}$. We then take $\alpha_0 \leftarrow \alpha_0/10$ and repeat the procedure until we reach the minimum of $I$ along the steepest descent direction $\hat v$ with the desired accuracy (tipically, we stop when we find the minimum for $\alpha_0 = \alpha_{min} = 10^{-10}$).
It should be noted that, depending on the local topology of $I(u)$, the algorithm might identify a local minimum for which, after the refinement of $\alpha_0$ takes place, we have both 
\begin{equation} \label{comparison1}
  I\left((w_1 + \alpha_k \hat v)  \left(\frac{\cdot}{t(  \alpha_k)}\right)\right) >  I\left((w_1 + \alpha_{k-1} \hat v)  \left(\frac{\cdot}{t( \alpha_{k-1})}\right)\right)
\end{equation}
and
\begin{equation} \label{comparison2}
   I\left((w_1 + \alpha_k \hat v)  \left(\frac{\cdot}{t( \alpha_k)}\right)\right) > I\left(w_1 +\alpha_{k+1} \hat v)  \left(\frac{\cdot}{t( \alpha_{k+1})}\right)\right).
\end{equation}
In fact, the algorithm has found a local maximum instead. The strategy in this case is to choose the function that gives the minimum on the right hand side of equations (\ref{comparison1}) and (\ref{comparison2}), and set it as $w_1^{new}$. The descent procedure would then carry on as described in (\ref{descent_num}).

\begin{remark}
	Our algorithm is not exempt from finding solutions other than the ground state.
	If the second case in Step 4 repeatedly leads to a curve over which the associated energy functional $I$ asymptotes a constant value, then Step 3 may give a steepest descent direction for which its norm goes to zero, and so we have found a critical point $w_c$.
	In our applications where the ground state is positive radially symmetric, it suffices to check if $w_c$ changes sign or not. At this point, we check if $w_c $ is a positive function for, if it is not, we return to Step 1 by taking an initial guess $w_0$ s.t. $I(w_0) < I(w_c)$ in order to proceed with the search for the ground state solution.
\end{remark}


\section{Numerical and parametrical study of the MMAP algorithm} \label{numerical}
 In order to assess the influence of the numerical parameters on the solution obtained by MMAP and on the behaviour of the algorithm, we will discuss in detail the influence of the numerical parameters on the convergence of the algorithm. The main numerical parameters which appear on MMAP are the following: the discretization size ($\Delta r$), the final resolution on the descent algorithm ($\alpha_{f}$), the size $R^*$ of the initial interval and the frequency of projections on $\mathcal{P}$ of the functions after the descent stage. To this end, we will choose $f(u)=u^3$ and $\lambda = 1.0$ in problem (\ref{problem_radial}), with standard set of parameters $M = 1001$, $\alpha_{min} = 10^{-10}$, $tol_{SOR} = 10^{-10}$ and $R^* = 1.0$ for all the simulations presented in this section, unless explicitly stated otherwise.

\subsection{Validation}

In order to verify that the implementation of MMAP is correct, we compare the result with the solution found via a different method (the mid-point method with Richardson extrapolation implemented in Maple 2018). The singularity of the equation is dealt with by assuming that the boundary conditions are defined as $u=\varepsilon_1$ at $r=100$, at which point we expect that $u$ is sufficiently close to zero, and $u'(\varepsilon_2)=\varepsilon_3$, where $\varepsilon_1= \varepsilon_2=\varepsilon_3 = 10^{-35}$. The results are plotted in Figure \ref{maple}. We observe a very good agreement with the result obtained by the aforementioned method.

\begin{figure}[h!]
  \centering
	\includegraphics[width=0.7\linewidth]{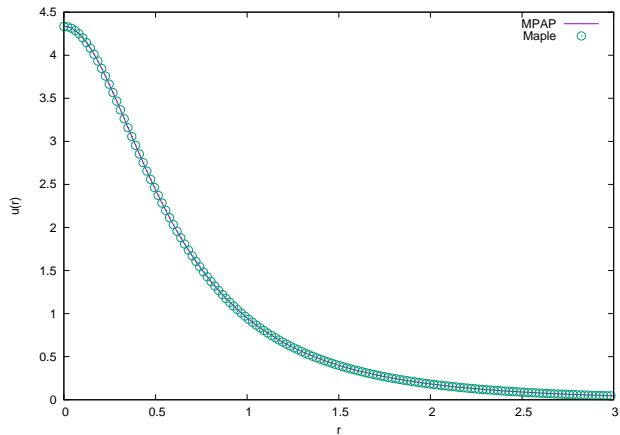}
	\caption{Comparison between the results given by MMAP (heavy line) and Maple (circles) for $f(u) = u^3$ and $\lambda = 1.0$, with standard set of parameters.}
	\label{maple}
\end{figure}

\subsection{Convergence}

We now assess the influence of the discretisation size $\Delta r$ on the results. Figure \ref{conv1} shows a comparison of the solution given by MMAP for several values of $\Delta r$. On Figure \ref{conv1} (left), we plot the solution obtained for different mesh sizes, corresponding to $\Delta r$ ranging from $0.02$ to $0.0004$, and we observe that no significant differences on the profile of the solution can be noticed. However, we do note that there is a difference on the tail of the solution when $\Delta r$ changes. Nevertheless, the differences are minor and due to the fact that the final length of $\Omega$ is actually calculated by the algorithm during the projection step, and will change depending on $M$ and on the initial value of $R^*$. This will be discussed further on Section \ref{size_domain}. Figure \ref{conv1} (right) indicates that this phenomenon does not compromise significantly the value of $I(u)$ for sufficiently large $M$. For $M > 400$, the differences among the solutions are negligible and the differences between consecutive curves and the values of $I(u)$ become smaller and smaller as $M$ grows.

\subsection{Initial size of the domain} \label{size_domain}
The point where the boundary condition at infinity is imposed at the beginning of the simulations defines the size $R^*$ of the domain $\Omega$  in which we define $w_0$. We have to choose $R^*$ sufficiently large, so that the numerical boundary condition is as realistic as possible, since we are looking for solution in $H_{rad}^1(\Rn)$. The effects of the choice of $R^*$ on the final results is assessed by measuring $||v||$ at the end of the simulations for different values of $R^*$. The results, shown in Figure (\ref{domainsize}), indicate that the smaller $R^*$, the larger the final $||v||$ will be. As $R^*$ increases, we observe that $||v||$ decays as ${R^*}^{-2}$ until it reaches a plateau at around $R^*=10$. For larger values of $R^*$, there is no significant change on the final value of $||v||$. This indicates that, for each choice of $\Delta r$, there is a minimum critical value of $R^*$ that must be chosen in order to achieve the best possible value of $||v||$ in the end of the simulations. Even further, for two different choices of $\Delta r$, as $R^*$ increases, we observe that the plateau is reached for the same $R^*$.

\begin{figure}[h!]
  \centering
      \includegraphics[width=0.49\linewidth]{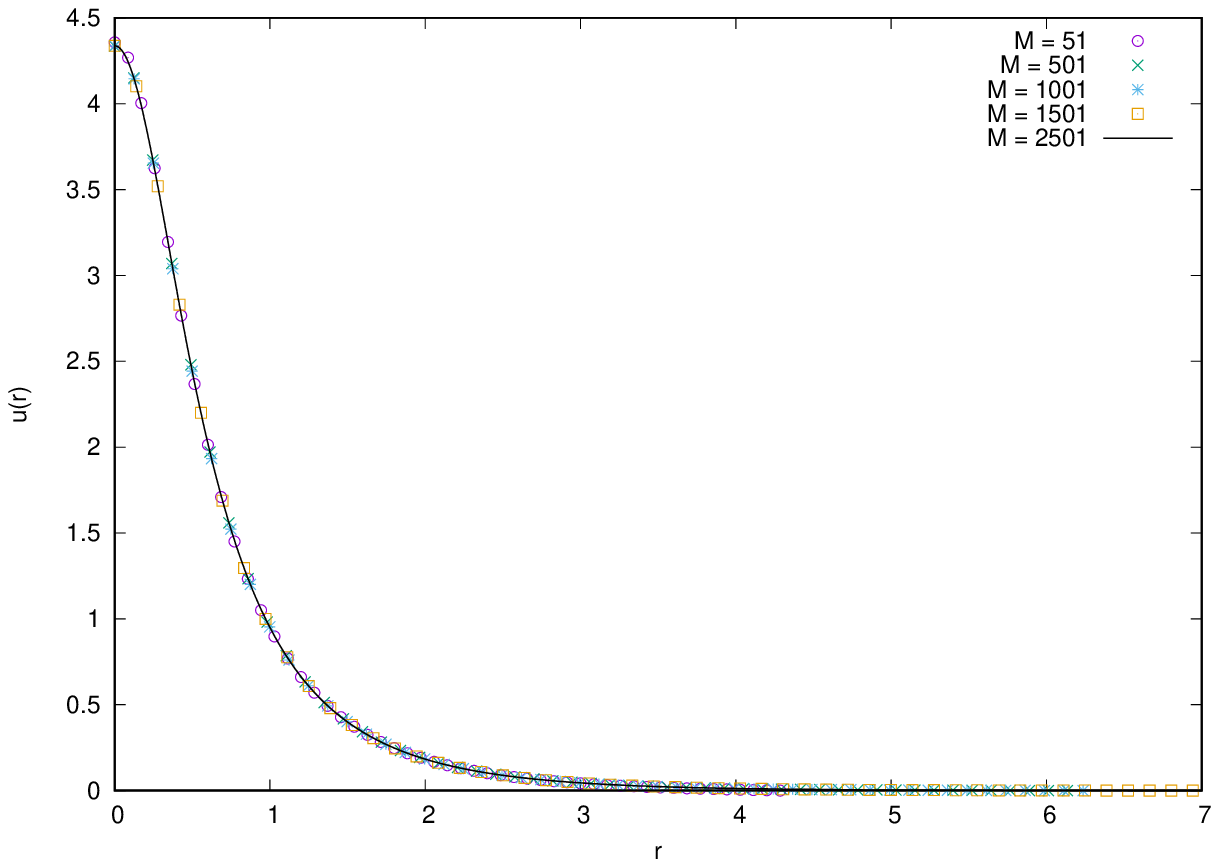}
      \includegraphics[width=0.49\linewidth]{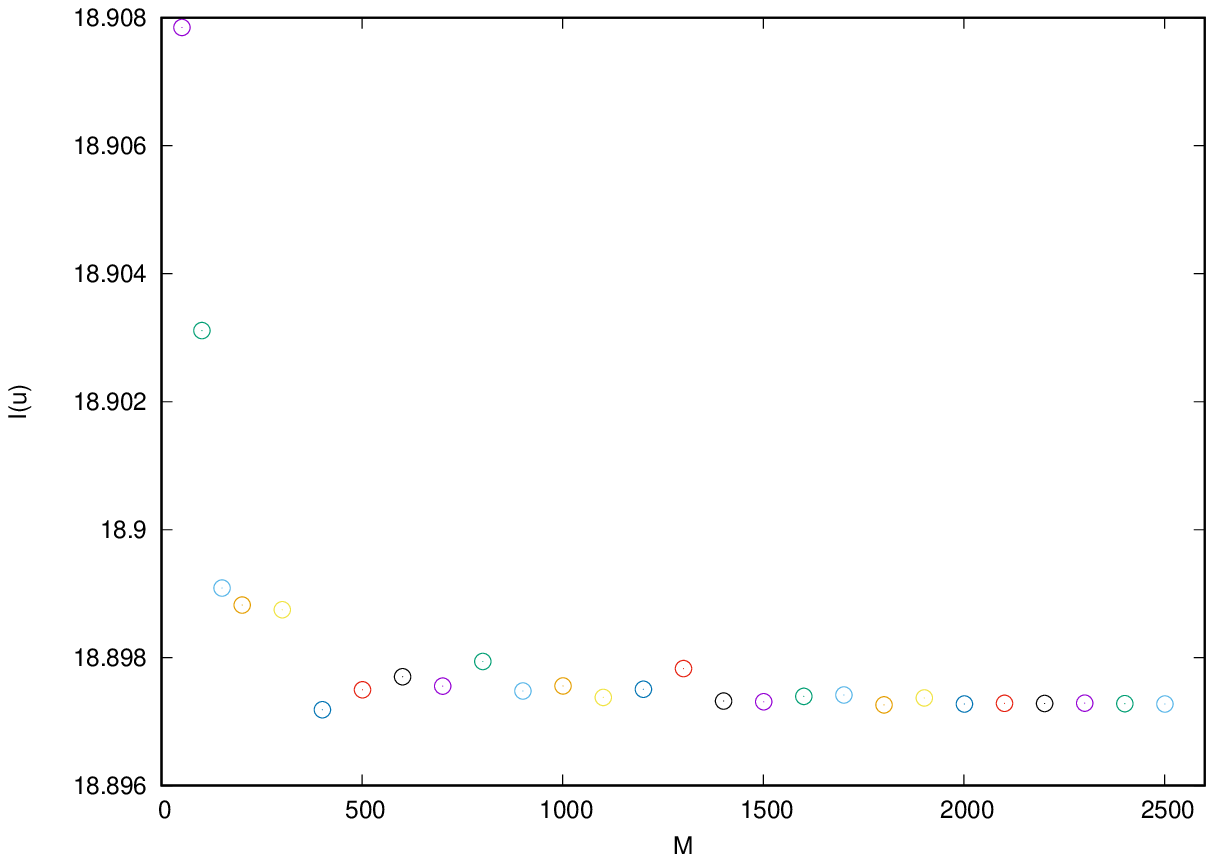}
	\caption{Comparison of the solution obtained by MMAP for different values of $M$ (left) and the values of $I(u)$ (right).}
	\label{conv1}
\end{figure}

\begin{figure}[h!]
	\centering
	\includegraphics[width=0.7\linewidth]{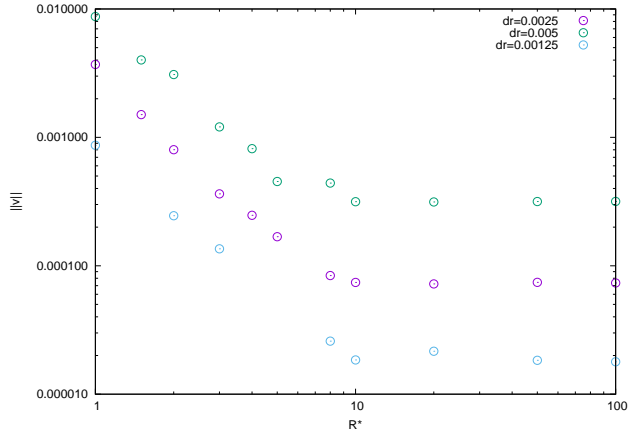}
	\caption{Influence of the initial size $R^*$ of the domain $\Omega$ versus the norm of the steepest descent direction obtained on the last iteration, for $\Delta r = 
	0.00125; 0.0025; 0.005$. The decay of $||v||$ to the minimum value obtained for large $R^*$ is roughly given by ${R^*}^{-2}$.}
	\label{domainsize}
\end{figure}

\subsection{Robustness}

Finally, we compare the results obtained by the standard choice of numerical parameters of our algorithm with a coarser mesh in which we have also reduced the values of $\alpha_{min}$ to $10^{-2}$ and the tolerance for the SOR algorithm $tol_{SOR}$ to the determination of the steepest descent direction also to $10^{-2}$. We observe very good agreement between the results, that is, the overall profile of the solution in the coarser approximation reproduces the shape of the refined solution, with the exception of the values close to $r =0$. In fact, the value of the functional $I(u)$ is only overestimated by around $0.1\%$ when the solution found by using the coarser parameters is used. This indicates that the algorithm is very robust and converges to the desired function even with very limited computational resources.

\begin{figure}[h!]
  \centering
	\includegraphics[width=0.7\linewidth]{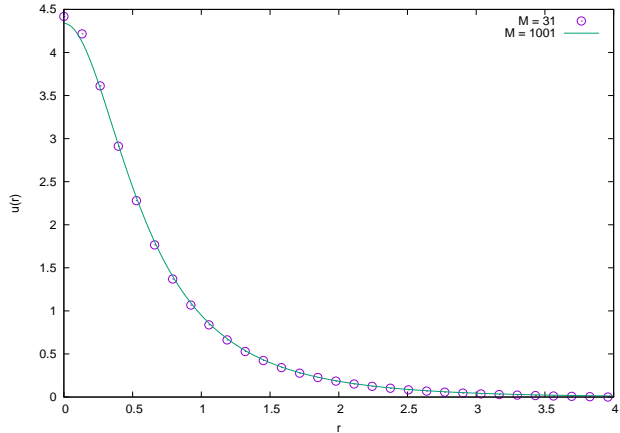}
	\caption{Comparison between the results given by the standard set of parameters (heavy line), and by the coarser mesh (circles), in which $M = 31$, $\alpha_{min} = 10^{-2}$ and $tol_{SOR} = 10^{-2}$.}
	\label{robustness}
\end{figure}

\subsection{Other remarks}
Solving the Poisson equation (\ref{steepest_radial}) is expected to be the most computationally expensive part of our algorithm and so, a parameter which must be given a good amount of significance is the tolerance for the convergence of the SOR. Since we are unaware of the local topology of the functional $I$, our initial guess $w_0$ from Step 1 might have high energy or be far from $\mathcal{P}$. Being so, at first, the tolerance $tol_{SOR}$ on the calculation of the steepest descent direction might be relaxed but, once we get close enough to the sought mini-max solution, this parameter must be refined.

For the choice of the initial guess $w_0$ in Step 1, the restriction $\int_{\Rn} G(w_0) > 0$ is mild compared to the initial guesses in the other algorithms in the literature.  

We note that Step 4, which involves the reprojection to $\mathcal{P}$ of the functions obtained during the descent phase of the algorithm, can be relaxed to a less computationally intensive version if we choose to perform the reprojections every $N_r$ steps. In fact, we have run several tests for $N_r$ ranging from 2 to 100 and no noticeable changes were observed neither on the shape of the solution nor on the value of $I(u)$ for the case $f(u) = u^3$, $\lambda = 1.0$ with the standard set of parameters.

\section{Applications to superlinear and asymptotically linear problems}
\subsection{The case $f(u) = u^3$ in $\mathbb{R}^3$}
For superlinear nonlinearities $|u|^p$, $1<p<2^* - 1$, the algorithms proposed prior to this work were able to tackle problem (\ref{prob}), which can also be managed by our algorithm. We can, apart from the validations performed in the previous section, assess its precision in calculating the maximum of the solution, which is attained in the origin, by recalling that the positive solution is radially symmetric and decreasing in the radial direction. Simple calculations show that
\begin{equation}\label{heights}
u_{\lambda}(r) = \lambda^{\frac{1}{p-1}} u_1(\sqrt{\lambda}\,r),
\end{equation}
is the positive solution of problem (\ref{prob}), with $f(u) = u^3$, where $u_1$ is the positive solution with $\lambda = 1.0$.

In Table \ref{heights_u3xx} we present the maximum heights $u(0)$ for several values of $\lambda$, obtained by our algorithm. On the other hand, assuming that the height of $u_1$ is given by our algorithm, that is, $u_1(0) = 4.33691$, we calculate $u_{\lambda}(0)$ for $\lambda = 0.1, 0.5, 2.0, 3.0$ using (\ref{heights}). The comparison of the heights $u(0)$ obtained numerically and the height $u_{\lambda}(0)$ obtained by (\ref{heights}) gives an error that is less than $0.1 \%$. Figure \ref{profiles_u3} shows the profiles of the solutions of problem (\ref{prob}) obtained by the algorithm for those values of $\lambda$.

\begin{table}[h]
	\centering
	\begin{tabular}{c|c|c|c|c}
		$\lambda$ & $u(0)$  & $||v||$    & $I(u)$   & error \\
		\hline
		0.1 & 1.37148 & $5.6 \cdot 10^{-4}$ & \hspace{5pt}5.97615  & $< 0.1\%$    \\
		0.5 & 3.06678 & $4.0 \cdot 10^{-4}$ & 13.36246 &  $< 0.1\%$  \\
		1.0 & 4.33691 & $6.0 \cdot 10^{-4}$ & 18.89734 & --  \\ 
		2.0 & 6.13321 & $7.7 \cdot 10^{-4}$ & 26.72488 & $< 0.1\%$  \\
		3.0 & 7.51153 & $9.3 \cdot 10^{-4}$ & 32.73110 & $< 0.1\%$  
	\end{tabular}
	\caption{Results for $u(0)$ for the case $f(u) = u^3$ obtained for different values of $\lambda$. In this table, we present the value of the norm of the steepest descent $||v||$ at the end of the calculations,  of $I(u)$ for the solution and the relative error of $u(0)$ with respect to the theoretical value $u_{\lambda}(0)$ in (\ref{heights}).}
	\label{heights_u3xx}
\end{table}

\begin{figure}[h!]
	\centering
	\includegraphics[width=1.0\linewidth]{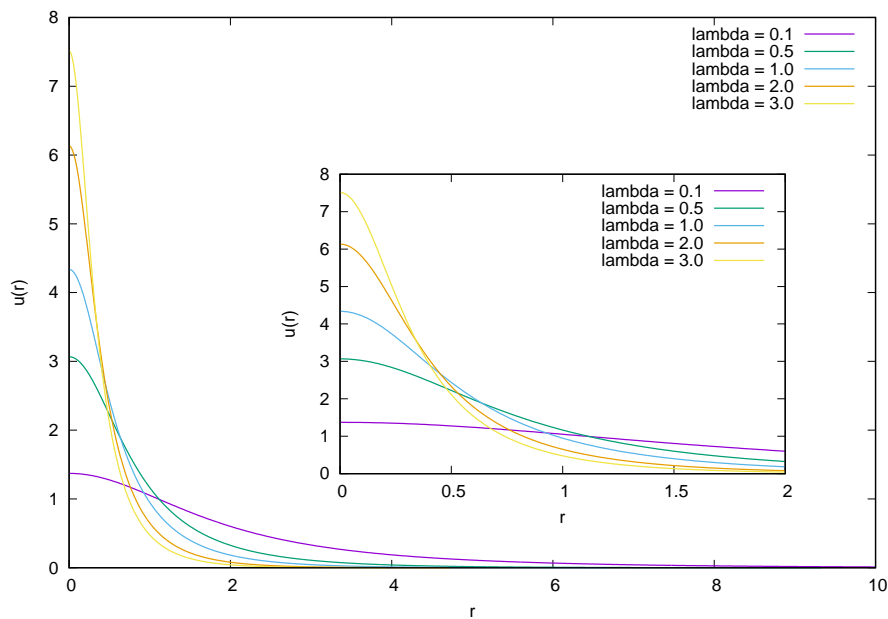}
	\caption{Profile of solutions for $f(u) = u^3$ for different values of $\lambda$.}
	\label{profiles_u3}
\end{figure}

\subsection{The case $f(u) = \displaystyle\frac{u^3}{1 + su^2}$ in $\mathbb{R}^3$}
The asymptotically linear problems $\displaystyle \frac {|u|^p}{1 + s|u|^{p-1}}$, $1<p<2^* - 1$, $0<\lambda s < 1$, satisfy the monotonicity condition $f(u)/u$ increasing for $u>0$ and so, could be handled by the algorithms in \cite{GJW} - since projections on the Nehari manifold rely on this hypothesis - but were not attempted. Using MMAP, we have found the ground state solution in the case $f(u) = \displaystyle \frac {u^3}{1+su^2}$. Figure \ref{surface_asym1} shows the solution for this nonlinearity with $\lambda = 1.0$ and $s = 0.5$. For reference purposes, we include on Table \ref{valores} the values of $u(0)$ for the positive solution $u$. Also, Figure \ref{energy_functional_asym1} shows the descending energy of the functional from the initial guess $w_0$, here chosen as $100 \, e^{-10 \, r^2}$, to the solution. 
For validation purposes, we present on Table \ref{tabela3} a list of values of the solution found for $\lambda =1.0$ and $s=0.5$.

\begin{figure}[h!]
	\centering
	\includegraphics[width=1.0\linewidth]{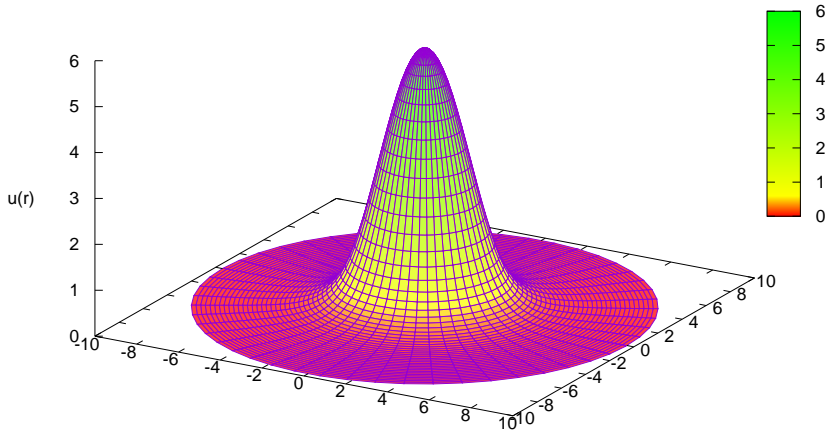}
	\caption{Surface plot of solution for $f(u) = \displaystyle\frac{u^3}{1 + su^2}$ with $\lambda = 1.0$, $s = 0.5$. $u(0) = 5.64139 $, $I(u) = 161.92929$, $||{v}|| = 2.5 \times 10^{-4}$.}
	\label{surface_asym1}
\end{figure}

\begin{figure}[h!]
	\centering
	\includegraphics[width=1.0\linewidth]{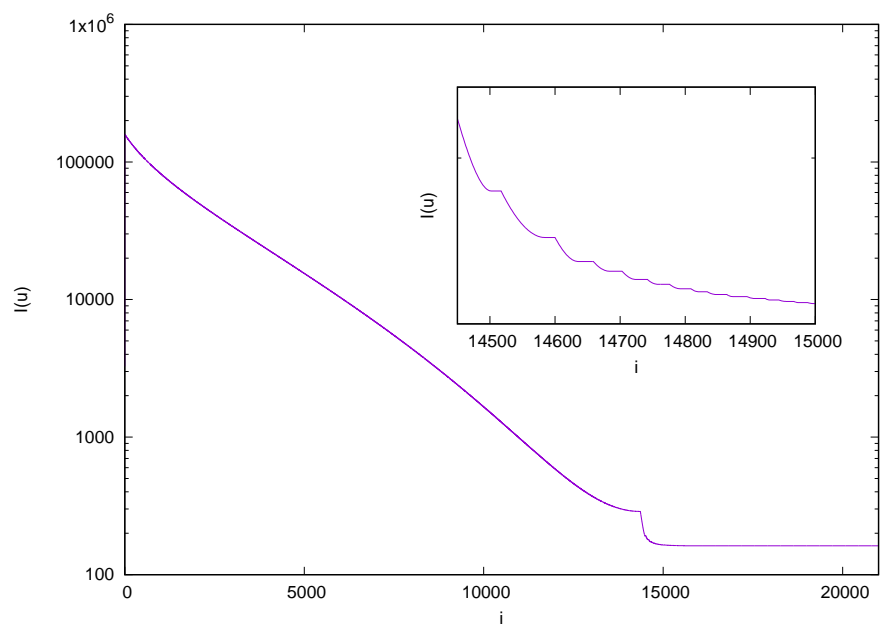}
	\caption{Descending energy of the associated functional along the iterations of the algorithm, for $f(u) = \displaystyle\frac{u^3}{1 + su^2}$ with $\lambda = 1.0$, $s = 0.5$. Logarithmic scale on the y axis.}
	\label{energy_functional_asym1}
\end{figure}

\begin{table}[h]
	\centering
	\begin{tabular}{c|c|cccccc}
          &     & \multicolumn{6}{c}{$\lambda$}     \\
          \hline
          \multirow{7}{*}{$s$} &     & 0.1 & 0.3 & 0.5 & 0.7 & 1.0 & 5.0 \\
          \hline
                     & 0.1 & 1.33183 & 2.23513 & 2.84300 & 3.34310 & 3.99690 & 12.61528 \\
                     & 0.3 & 1.29034 & 2.18677 & 2.87000 & 3.51098 & 4.50062 & --       \\
                     & 0.5 & 1.27125 & 2.22308 & 3.05319 & 3.94794 & 5.64139 & --       \\
                     & 0.7 & 1.26344 & 2.29849 & 3.33592 & 4.65516 & 8.08286 & --       \\
                     & 1.0 & 1.26374 & 2.46503 & 3.98912 & 6.76196 &   --    & --       \\
                     & 5.0 & 1.78424 &   --    &   --    &    --   &   --    & --    
        \end{tabular}
	\caption{ Values of $u(0)$ obtained for the case $f(u) = \displaystyle\frac{u^3}{1 + su^2}$ for several combinations of $\lambda$ and $s$. Note that we can only obtain solutions when $\lambda s < 1$. $M=3501$.}
	\label{valores}
\end{table}

\begin{table}[h!]
	\centering
	\begin{tabular}{c|c||c|c||c|c||c|c}
          $r$ & $u(r)$ &  $r$ & $u(r)$ & $r$ & $u(r)$ & $r$ & $u(r)$ \\
	  \hline
      0.000 & 5.64139  &  2.004 & 2.99197 &5.005 & 0.11309       &8.007 & 1.86676 $\times 10^{-3}$ \\
	  0.100 & 5.63348  &  2.205 & 2.58907 &5.207 & 8.88979	     &8.208 & 8.34267 $\times 10^{-4}$ \\
	  0.201 & 5.60837  &  2.608 & 1.84032 &5.601 & 5.56388 $\times 10^{-2}$  &8.300 & 3.91292 $\times 10^{-4}$ \\
	  0.302 & 5.56672  &  3.002 & 1.23610 &6.003 & 3.45536 $\times 10^{-2}$  &8.351 & 1.55421 $\times 10^{-4}$ \\
	  0.402 & 5.50879  &  3.203 & 0.98899 &6.204 & 2.72241 $\times 10^{-2}$  &8.376 & 3.87317 $\times 10^{-5}$ \\
	  0.604 & 5.34578  &  3.605 & 0.61708 &6.607 & 1.68230 $\times 10^{-2}$  &8.384	& 0.000000 \\
	  1.006 & 4.84857  &  4.007 & 0.37890 &7.000 & 1.03282 $\times 10^{-2}$  &		&  \\
	  1.199 & 4.54191  &  4.201 & 0.29952 &7.202 & 7.93701 $\times 10^{-3}$  & 		&  \\
	  1.601 & 3.80120  &  4.603 & 0.18367 &7.604 & 4.38170 $\times 10^{-3}$  &      &  	
	\end{tabular}
	\caption{ Values of $u(r)$ for $f(u) = \displaystyle \frac{u^3}{1 + su^2}$ for several $r$ with $\lambda = 1.0$, $s = 0.5$.}
	\label{tabela3}	
\end{table}

\section{Enhancement for more general nonlinearities}

The real improvements of our algorithm compared to others in the literature are presented in the next two examples. In order to obtain the positive ground state solution of (\ref{prob}), depending on the nonlinear term $f(u)$ the algorithm MMAP
is applicable and gives the correct solution, whereas other existing algorithms cannot be applied either because it requires unique projections on the Nehari manifold \cite{GJW} or because superquadratic conditions on the nonlinearity $f$ are assumed \cite{YP}.



\subsection{Example where $I(tu)$ has two maxima for $t>0$ }
This example illustrates a situation where the functional $I$ evaluated in the direction $tu$, for $t \in \mathbb{R}$, has at least two maximum values at $t_1$ and $t_2$, for instance, and hence the algorithm MPA developed by 
Chen, Ni and Zhou in \cite{GJW}, which takes the unique projection on the Nehari manifold on the direction of the vector $u$ (Step 3), does not work. 

Choosing $F(u) = Bu^3 - Cu^4 + Du^5$ in (\ref{functional}), and so

\begin{equation}\label{nonlinearity_nehari}
	f(u) = 3Bu^2 - 4Cu^3 + 5Du^4,
\end{equation}
with $\lambda = 3.0$, and taking
\begin{equation*}
u(r) = 
\left\{\begin{array}{rcll}
\displaystyle&\frac{1}{\sqrt{4\pi}}&, \qquad \qquad \quad |r| \le R\\
\displaystyle&\frac{1}{\sqrt{4\pi}}& e^{-|R-r|}, \qquad |r| \ge R
\end{array}\right.
\end{equation*}
with $R \approx 3.075$, $A = \displaystyle\frac{||u||^2}{2}$ and positive constants $B, C$ and $D$ such that

\begin{align*}
	I(tu) &= t^2 \displaystyle\frac{||u||}{2} - \int F(tu) \\
		  & = t^2 A - B t^3 \int u^3 + C t^4 \int u^4 - D t^5 \int u^5  \\
		  &= -t^5 + (5 + \sqrt{5})t^4 - 2(4+\sqrt{5})t^3 + 4(1+\sqrt{5})t^2
\end{align*}
gives rise to an example for $I(tu)$ having two maxima. Figure \ref{u_counter} shows $u(r)$ and $I(tu)$. Those two maxima are given by $I(r_1) = I(r_2) = \displaystyle\frac{128}{25\sqrt{5}}$. The profile of the solution for problem \ref{prob} with $f(u)$ as in (\ref{nonlinearity_nehari}), with $\lambda = 3.0$ solved by MMAP is shown in Figure \ref{counter_nehari_sol}.

\begin{figure}[h!]
	\centering
	\includegraphics[width=0.49\linewidth]{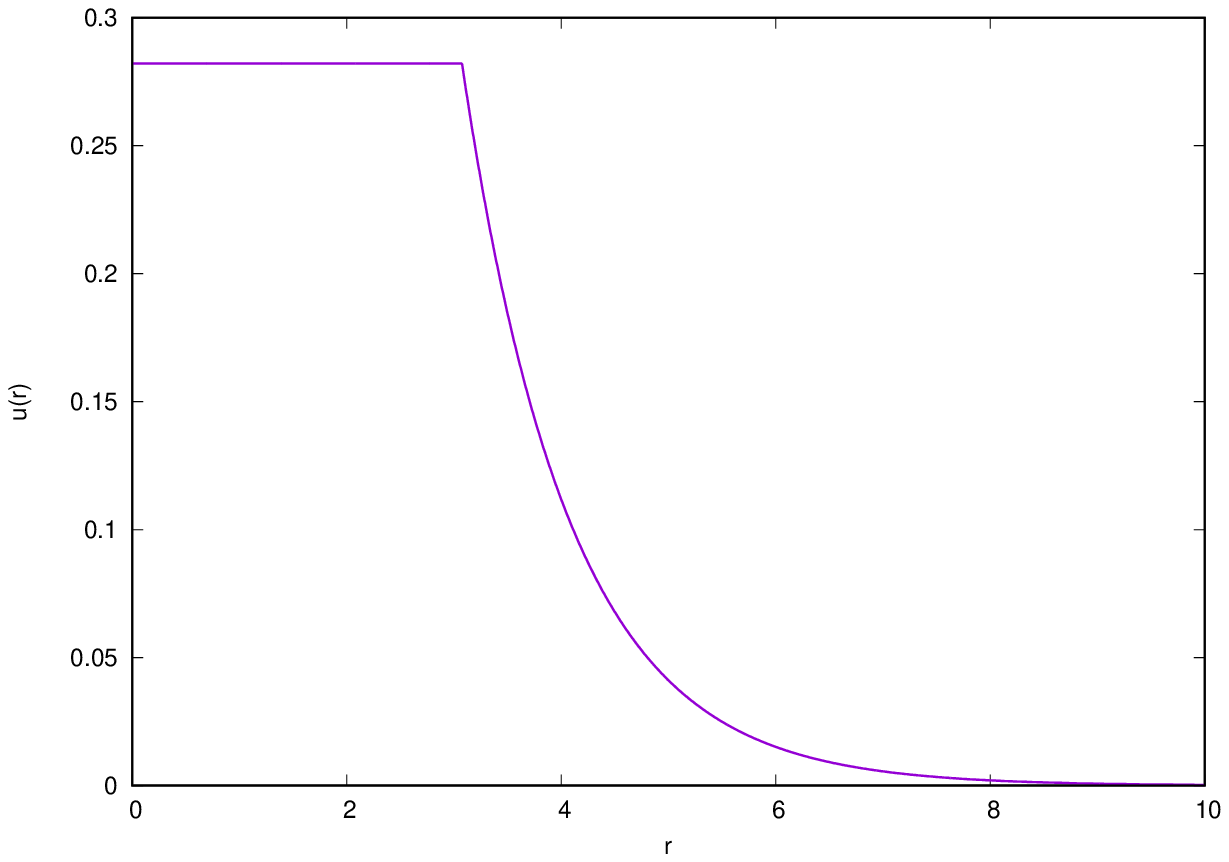}
	\includegraphics[width=0.49\linewidth]{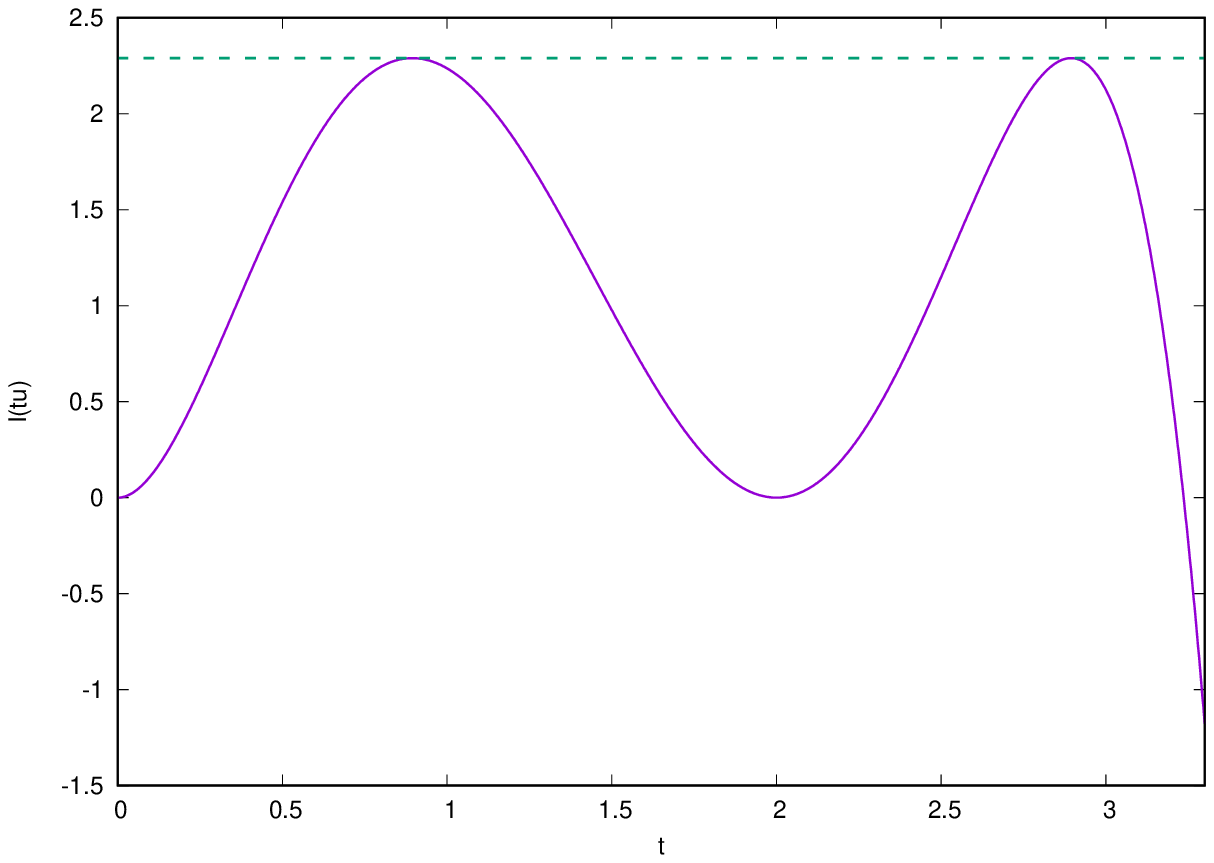}
	\caption{An example of function $u$ (left) for which $I(tu)$ has two maxima (right), with $f(u)$ as in (\ref{nonlinearity_nehari}).}
	\label{u_counter}
\end{figure}

\begin{figure}[h!]
	\centering
	\includegraphics[width=0.7\linewidth]{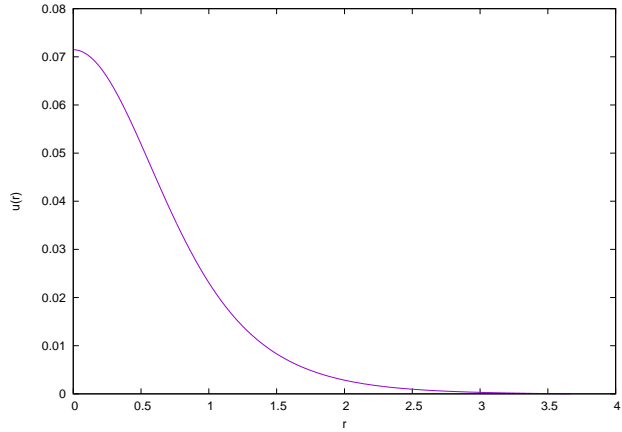}
	\caption{Profile of solution for the nonlinearity (\ref{nonlinearity_nehari}) for $\lambda = 3.0$.}
	\label{counter_nehari_sol}
\end{figure}

\section{Concluding remarks}

The algorithm presented in this paper is based in a novel approach of finding a critical point of a functional associated to the Euler equation, which may model Physical problems, by constrained minimization method in the appropriate Pohozaev manifold. The main advantage is that it can tackle asymptotically linear as well as superlinear problems with no assumption of monotonicity on $f(u)/u$. This improves previous results by solving for those problems already studied and complementing with new problems which could not be treated by the preceding algorithms in the literature.

The example

\begin{equation}
	f(u)= \displaystyle\frac{u^7 - \frac{5}{2}u^5 + 2u^3}{1 + su^6},
\end{equation}
shown in Figure \ref{monotonicity_f} (left), does not satisfy the monotonicity condition of $f(u)/u$, shown in Figure \ref{monotonicity_f} (right), increasing in the variable $u$, for $u>0$.
However, projections on the Pohozaev manifold can be performed, hence MMAP can be applied.

\begin{figure}[h!]
  \centering
      \includegraphics[width=0.49\linewidth]{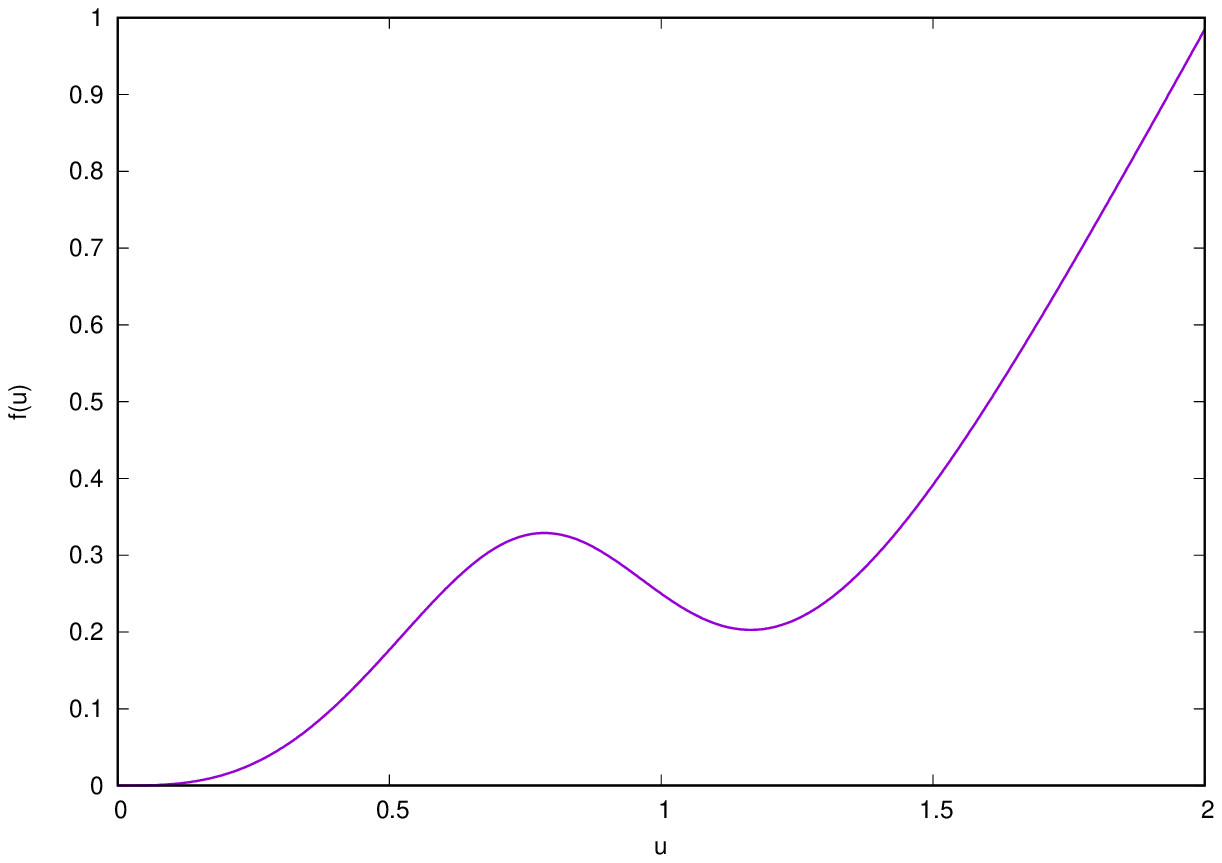}
      \includegraphics[width=0.49\linewidth]{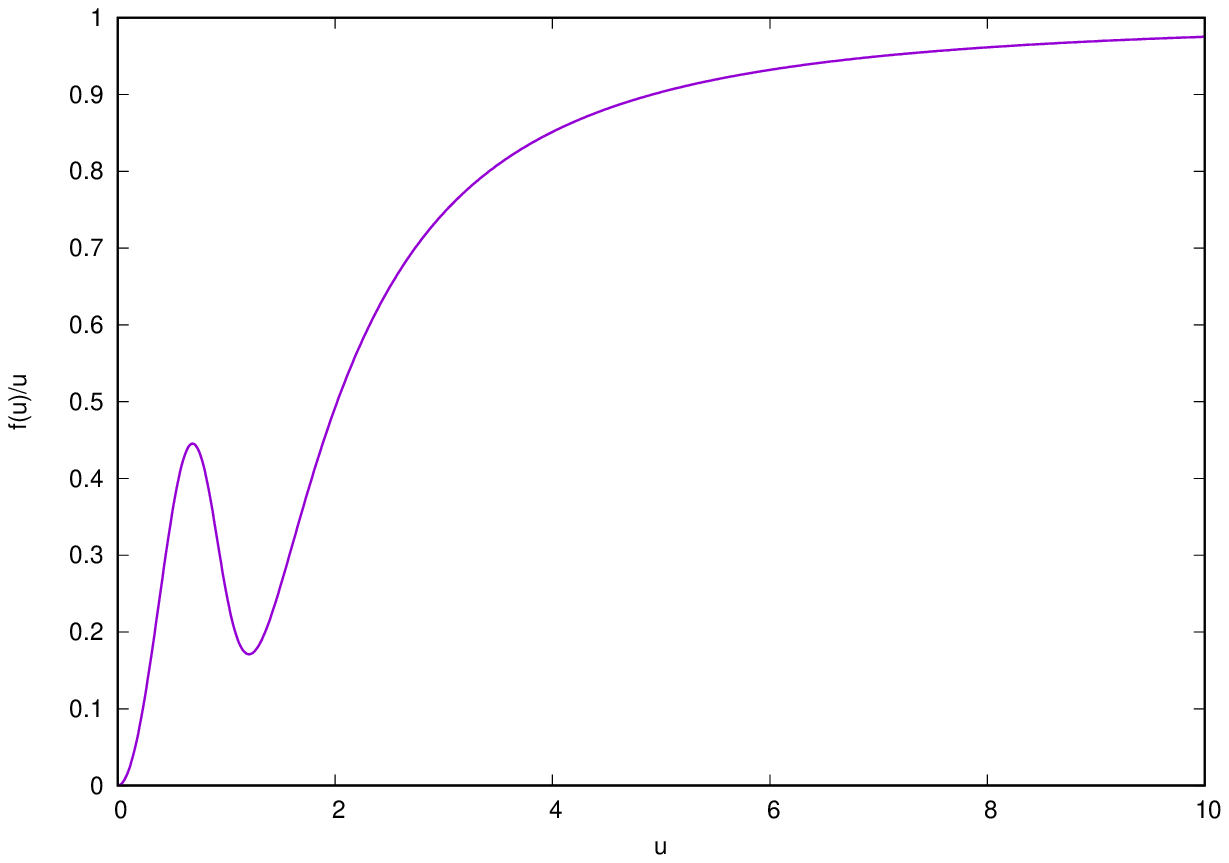}
	\caption{Example of a nonlinearity for which the monotonicity condition does not hold. $f(u)$ (left) and $f(u)/u$ (right).}
	\label{monotonicity_f}
\end{figure}

The theoretical backing of this algorithm is the variational method where the associated functional $I$ is defined on the Hilbert space $H^1(\Rn)$, which is continuously embedded in $L^{2^*}(\Rn)$. Hence, critical and supercritical nonlinear terms, 
$\lim_{u \to +\infty} f(u)/u^p=+\infty$, with $p\geq2^*-1$, cannot be accessed.
Even further, in using ordinary differential equations for finding radial solutions of the problem, another method would be needed in order to arrive at the ground state solution.

\clearpage




\end{document}